\documentclass{article}

\usepackage{indentfirst}
\usepackage{graphicx}
\usepackage{color}
\usepackage{amssymb}
\usepackage{amsmath}
\usepackage{amsthm}
\usepackage{float}
\usepackage{bm}
\usepackage{cases}
\usepackage{mathrsfs}
\usepackage{hyperref}
\usepackage{cleveref}
\usepackage{epstopdf}

\usepackage{algorithm}
\usepackage[noend]{algpseudocode}

\newtheorem{thm}{Theorem}

\DeclareGraphicsExtensions{.png,.jpg,.gif,.pdf,.eps,.ps}

\numberwithin{figure}{section}
\numberwithin{equation}{section}
\numberwithin{table}{section}
\allowdisplaybreaks

\begin{document}
\title{Energetic Stable Discretization for Non-Isothermal  Electrokinetics Model}
\author{Simo Wu\thanks{Department of Mathematics, Pennsylvania State University, University Park, PA 16802, USA (szw184@psu.edu).}, Chun Liu\thanks{Department of Applied Mathematics, Illinois Institute of Technology, Chicago, IL 60616, USA (cliu124@iit.edu).}, Ludmil Zitakanov\thanks{Department of Mathematics, Pennsylvania State University, University Park, PA 16802, USA (ludmil@psu.edu).}}

\pagestyle{myheadings} \markboth{Numerics on Non-Isothermal Electrokinetics}{ S. Wu, L. Zitakanov, C. Liu} \maketitle

\abstract{We propose an edge averaged finite element(EAFE) discretization to solve the Heat-PNP (Poisson-Nernst-Planck) equations approximately. Our method enforces positivity of the computed charged density functions and temperature function. Also the thermodynamic consistent discrete energy estimate which resembles the thermodynamic second law of the Heat-PNP system is prescribed. Numerical examples are provided.}

 \section{Background}

 The study of thermodynamic properties of complex fluid involving electrical solvent is of interest of many biological and physiological applications~\cite{Nithiarasu:2018,Knox:1994,Eleuteri:2015}. For example, electrokinetic microfluidic devices have been widely used in biomedical and biotechnological applications and chemical synthesis, which can be used for pumping and controlling the liquid flow in microfluidic systems or separating constituents suspended in a liquid~\cite{Knox:1994, Ramos:2006}. While low thermal conductivity devices are widely used, temperature effect should be includeded as it affects the electrical conductivity and can produce substantial heat driven flow. In addition, the interaction between heat flux and electric field could result in Joule heating effect as a byproduct of presense of electric field ~\cite{Sanchez:2013,Grushka:1989aa}. This effect is generated by the ohmic resistance of the electrolyte subjected to electric current and, clearly, the traditional isothermal models aer not adequate in such situation.
More recent works account for thermal effects and show that these interactions can be modeled in a thermodynamic consistent way (see~\cite{Liu:2017}).

Another important pool of applications stems from biology, namely, modeling ion channels. Such channels are formed by charged walls and play a fundamental role in controlling and regulating the nervous system.  In addition to the three main types of ion channels: voltage gated, ligand gated and chemically gated ion channels, temperature controlled ion channels are also been found and studied in the literatude (see~\cite{Cesare:1999,Reubish:2009}). For example, six members of the mammalian Transient receptor potential ion channels (TRP) respond to varied temperature thresholds. TRPV1, TRPV2, TRPV3, and TRPV4 are heat activated, whereas TRPM8 and TRPA1 are activated by cold~\cite{Reubish:2009}.

Despite there is vast amount of works devoted on related experimental studies\cite{Grushka:1989aa,Knox:1994}, mathematical modeling\cite{Tang:2003,Petersen:2004} and simulation\cite{Xuan:2008}, \cite{Shamloo:2015}, 
thermodynamically consistent models are much less studied and developed.

To begin the discussion of the non-isothermal model we first consider
the isothermal case and the Poisson-Nernst-Planck (PNP) system. This model
is widely used to describe the transport of charged particles ~\cite{Qiao:2014, Im:2002, Gillespie:2002, Guo-Wei:2012}with low
concentration and couples the well-known drift-diffusion equations for
the concentration with the Poisson equation for the electrostatic
potential with appropriate initial and boundary conditions~\cite{Liu:2014}:
  \begin{align*}
 -\nabla\cdot(\epsilon\nabla\phi) & = e_c \sum_{i =1}^{N}q_{i}\rho_{i} \tag{Poisson Equation}\\
 \frac{\partial\rho_{i}}{\partial t} &= -\nabla\cdot \vec{J}_{i} \tag{Drift Diffusion equation} \\
 \vec{J}_{i} & = -D_{i}\nabla\rho_{i}-q_{i}\mu_{i}\rho_{i}\nabla\phi\tag{Flux},
 \end{align*}
Here $D_{i} > $  0 is the diffusivity, $q_{i}$ is the valence number, and $\mu_{i}$ is the mobility of the $i$-th ion species.
The PNP system models the interaction of N ionic species through an electrostatic field, usually $N \geq 2$. $\phi$ is the electrostatic potential, $\rho_{i}$ stands for the charge density of the ith species.

Both ion channels and electroosmosis can be modeled by taking the background fluid into consideration and  coupling the PNP system with the Navier-Stokes equation. The energetic variational formulation which deals with the isotropic case for the resulting system from a mechanical prospective is well elaborated in\cite{Xu:2014ab, Horng:2012aa}. The energy law for this system, as stated in~\cite{Xu:2014ab} is:

\begin{eqnarray*}
  \lefteqn{
    \frac{d}{dt}\{ \int_{\Omega} \frac{1}{2}\rho_f |\vec{u} |^{2} +\sum_{i = 1}^{N} \rho_{i}(\rho_{i}-1)+\frac{\epsilon}{2}|\nabla\phi |^2 dx +\int_{\Gamma_{R}} \frac{\kappa}{2}| \phi|^2 ds\}
  }\\
  &=& - \int_{\Omega} \frac{\mu}{2} |\frac{{\nabla u + \nabla u }^{T}}{2} |^2 + \sum_{i = 1}^{N} D_{i}\rho_{i}|\nabla(\log \rho_{i}+q_{i}\phi)|^2 dx,
\end{eqnarray*}
Here, $\Gamma_{R}$ denotes the part of boundary where the Robin boundary condition: $\epsilon \nabla\phi \cdot \vec{n} +\kappa \phi = C$ is imposed.
The total energy consists of the sum of the kinetic energy of the background fluid, the electro potential energy and the free energy of each of the species. The dissipation of the total free energy is driven by the viscosity and diffusion. This coupled system can be derived through energetic variational approach (EVA, introduced in~\cite{Eisenberg:2007}) and is as follows:
\begin{align*}
 -\nabla\cdot(\epsilon\nabla\phi) & = e_c \sum_{i =1}^{N}q_{i}\rho_{i} \tag{Poisson equation}\\
\frac{\partial\rho_{i}}{\partial t} &= -\nabla\cdot \vec{J}_{i} \tag{Drift diffusion equation} \\
 \vec{J}_{i} & = -D_{i}\nabla\rho_{i}-q_{i}\mu_{i}\rho_{i}\nabla\phi-\rho_{i}\vec{u} \tag{Flux}\\
 \rho_{f}(\vec{u}_{t}+\vec{u}\cdot \nabla\vec{u})+\nabla p& = \nabla \cdot \frac{{\nabla u + \nabla u }^{T}}{2} -  \sum_{i = 1}^{N}q_{i}\rho_{i}\nabla\phi \tag{NS equation}\\
 \nabla\cdot\vec{u} &= 0 \tag{Incompressibility}
\end{align*}
As is evident from the equations above, this model does not take temperature effect into consideration. 

A number of numerical solution approaches for the PNP system and relevant mathematical models models have been considered in the literature. For example, A. Flavell \textsl{et al.} (2014) \cite{Flavell:2014} applied a conservative finite difference scheme which achieves second-order accuracy in both space and time and also conserves total concentration for each ion species.  C. Liu \textsl{et al.} (2015) \cite{Liu:2015} used an energetically stable finite element method for the PNP-NS(Poisson-Nernst-Planck-Navier-Stokes) system. D. Xie \textsl{et al.} (2016) \cite{Xie:2016} used a non-local finite element method for the PB (Poisson-Boltzmann) equation to tackle the problem of solution singularity caused by point charge term.

In order to capture the temperature change \cite{Feireisl:2007}, \cite{Feireisl:2009} considered a thermodynamic approach for incompressible Newtonian fluid coupled with temperature model. The resulting model is based on the first and second law of thermodynamics and in accordance with
the Fourier's Law assumes that  the internal energy is proportional to the temperature and the heat flux is proportional to the temperature gradient. The equation for the temperature is derived through energy conservation and the whole system satisfies the inequality constraint for entropy production and this is in agreement with the second law of thermodynamics.
 
 Some works in biology~\cite{Tang:2003},\cite{Petersen:2004}, take both temperature and electrokinetics into consideration by a simple coupling which adds electrokinetic force into the Navier-Stokes equation and use PNP or Poisson-Boltzmann to model the electric field. The temperature effects are incorporated by including  the Joule heat force term into the heat equation. One issue with this model is that it is unclear  what the energy associated with this system, and, moreover, it is not evidennt whether the second thermodynamics law is violated by the model. 

In our study, we are using a more consistent model to capture the interactions mentioned above. According to first law of thermodynamics, heat and work can convert to each other, so it is natural to consider this energy  interchange in order to derive the equation for the temperature. By the second law of thermodynamics, the conversion between heat transfer and work follows the rule that keeps the entropy increasing. We follow the ideas in~\cite{Liu:2017} and add this as an  inequality constraint for the system. As it turns out, maintaining such inequality constraints is also crucial for the stability of the discretized model. This motivated the choice of discretization and we have applied the Edge Average Finite Element (EAFE, see~\cite{Xu:1999}) discretization, which satisfies the discrete maximum principle for the temperature. The discrete maximum principle turns out to be the key to numerical stability and one of the key results that we present is the energy estimate satisfied by the numerical model.

The paper is organized as follows. We introduce the Heat-PNP equations in section 2 and the corresponding energy law in section 3. In section 4 we propose our discretization and prove an energy estimate for the discretized Heat-PNP system. In section 4 some numerical experiments are provided to validate the stability of our numerical scheme. We also  provide numerical tests which show consistency with qualitative phenomena observed experimentally~\cite{Nithiarasu:2018}. 

\section{The Heat-PNP equations and their discretization}

As it is well known~\cite{kirby:2010} the PNP system is as follows:
\begin{equation}
\begin{cases}
\displaystyle \frac{\partial}{\partial t} \rho_i = \nabla \cdot \left(D_{i} \left( \nabla\rho_i + \frac{z_{i}e}{k_{B}T}\rho_i\nabla\phi \right)\right) ,      i = 1,...,N, \\
\displaystyle \nabla\cdot\left(\epsilon \nabla\phi \right) = -\left(\rho_0 +\sum_{i=1}^{N}z_{i}e\rho_i\right)
\end{cases}
\end{equation}
We treat the evolution equations for ions' concentration as a system of continuity equations. After introducing the flux variable for each ion species, this sytem can be written as:
\begin{equation}
\begin{cases}
\displaystyle \frac{\partial\rho_{i}}{\partial t} = -\nabla\cdot \vec{J}_{i} , \\
\displaystyle \vec{J}_{i}  = -D_{i}\nabla\rho_{i}- \frac{z_{i}e}{k_{B}T}\rho_{i}\nabla\phi, \\
\displaystyle \nabla\cdot\left(\epsilon \nabla\phi \right) = -\left(\rho_0 +\sum_{i=1}^{N}z_{i}e\rho_i\right)
\end{cases}
\end{equation}
We use the same fashion to rewrite the PNP system with temperature throughout this paper.

We consider the the temperature PNP system in $\Omega \subseteq \mathbb{R}^{n},n = 2,3$ without background fluid. According to \cite{Liu:2017}, this system has the form:
\begin{equation}
\begin{cases} 
\displaystyle \frac{\partial}{\partial t} \rho_i + \nabla \cdot (\rho_i \vec{u_i})=0 ,      i = 1,...,N, \\
\displaystyle \nu_i \rho_i \vec{u_i} = - k_B \nabla (\rho_i T) - z_i e \rho_i \nabla \phi,      i = 1,...,N,  \\
\displaystyle -\nabla \cdot \epsilon \nabla \phi =\sum_i \rho_i z_i e+\rho_f. \\
\displaystyle \left(\sum_{i=1}^N  k_B C_i \rho_i\right) \frac{\partial T}{\partial t} + \left(\sum_{i=1}^N  k_B C_i \rho_i \vec{u_i} \right) \cdot \nabla T + \left(\sum_{i=1}^N k_B \rho_i  \nabla \cdot \vec{u_i}\right) T \nonumber =\\  \nabla \cdot k \nabla T +  \sum_{i=1}^N \nu_i\rho_i |\vec{u_i}|^2 +q,
\end{cases}
\label{temp_PNP}
\end{equation}
Here, $\phi$ is the electrostatic potential , $\rho_{i}$ stands for the charge density of the ith species, T is the temperature function and $\vec{u_i}$ is the macroscopic velocity of the ith species.  All these variables are space time functions defined on $\Omega \times [0,T]$. $D_{i} > $  0 , $q_{i}$, and $\mu_{i}$ are correspondingly the diffusivity constant, the valence number and the constant mobility of the ith ion species.

\subsection{Initial and Boundary Conditions}

For the system~\eqref{temp_PNP}, the initial conditions at \(t=0\) are 
\[
T(x,0) = T_{0}(x), \quad
\rho_i(x,0) = \rho_{i,0}(x), \quad i = 1,\ldots,N,
\quad x\in  \Omega,
\]

The boundary conditions are of different types:
We have homogeneous no-flux boundary conditions for each of the concentrations of ion species,
\[
\rho_{i}\vec{u_{i}}\cdot \vec{n}  = 0           \text{         on      }  \partial\Omega.
\]
For the Poisson equation, we partition the boundary into disjoint parts: $\partial\Omega = \Gamma_{D}\cup\Gamma_{N}\cup\Gamma_{R} $
\begin{align*}
\phi &= \delta V \text{               on     } \Gamma_{D},\\
\epsilon\nabla\phi\cdot \vec{n} &= E\text{               on     } \Gamma_{N},\\
\epsilon\nabla\phi\cdot \vec{n} + \kappa\phi&= C\text{               on     } \Gamma_{R},\\
\end{align*}
For the temperature equation we have Dirichlet boundary conditions
which, as we show later help us in deriving the discrete energy law.
\[
T =\delta T  \text{      on       } \partial \Omega.
\]

\section{Energy of the Heat-PNP system} 
According to \cite{Liu:2017}, the Heat-PNP system satisfies simultaneously first and second energy laws of thermodynamics, which are:
\begin{equation}
\begin{cases}
\displaystyle  \frac{d}{dt}(U+K) = \int_\Omega \left(q+ \sum_{i=0}^N \rho_i \frac{\partial \psi}{\partial t} \right)d\mathbf{r} + \int_{\partial \Omega} k\nabla T d\mathbf{r},\\
\displaystyle   \frac{d}{dt} S = \Delta + \int_\Omega \frac{q}{T} d\mathbf{r} + \int_{\partial \Omega}  \frac{k\nabla T}{T} d\mathbf{r}.
\end{cases}
\label{energy_law_g}
\end{equation}
These laws involve the  internal energy functional $U$,
the kinetic energy functional $K$ of the system,  and the entropy functional
$S$ of the whole system. In addition, $\psi$ is the applied electric field, which, for simplicity, we assume to be independent of time in our system here. 

More precisely, if we substitute the predefined entropy functional and entropy production functional for Heat PNP system we obtain:
\begin{equation}
\begin{cases}
\displaystyle S = \sum_{i = 1}^{N}\int_\Omega k_{B} \rho_i(\log \rho_i -C_i \log T-C_i) dx\\
\displaystyle \Delta = \sum_{i = 1}^{N} \int_{\Omega} \frac{\nu_{i}\rho_i |\vec{u_i}|^2}{T}+\frac{1}{k}\left|\frac{k\nabla T}{T}\right|^2  dx
\end{cases}
\end{equation}
Next, we use the divergence theorem to derive the second law for our system in terms of the unknon variables:
\begin{equation}
\frac{d}{dt}\sum_{i = 1}^{N}\int_\Omega k_{B} \rho_i(\log \rho_i -C_i \log T-C_i) dx = -\sum_{i = 1}^{N} \int_{\Omega} \frac{\nu_{i}\rho_i |\vec{u_i}|^2}{T}+\frac{1}{k}|\frac{k\nabla T}{T}|^2  - \nabla\cdot(\frac{k\nabla T}{T}) dx
\label{energy_law_T}
\end{equation}

In fact, to derive the second energy law of thermodynamics from the equation set is straightforward. We multiply by $\frac{1}{T}$ both sides of the heat equation \eqref{temp_PNP} and integrate over the domain. For the left hand side we then we have
\begin{eqnarray*}
\lefteqn{\int_\Omega \sum_{i = 1}^{N}k_{B}C_{i}\rho_{i} \frac{\partial T}{\partial t}\cdot\frac{1}{T}+\sum_{i = 1}^{N}k_{B}C_{i}\rho_{i}\nabla\cdot (u_{i}T)\cdot\frac{1}{T}+ \sum_{i = 1}^{N}k_{B} (1-C_{i})(\rho_{i})\nabla \cdot u_{i} dx}  \\
& = & \int_\Omega \sum_{i = 1}^{N}k_{B}C_{i}\rho_{i} \frac{\partial \log  T}{\partial t} + \int_\Omega \sum_{i = 1}^{N} k_{B}C_{i}(\frac{\nabla(\rho_{i}u_{i}T)}{T}-u_{i}\nabla\rho_{i})+ \sum_{i = 1}^{N}k_{B} (1-C_{i}) (\nabla(\rho_{i}u_{i}-u_{i}\nabla\rho_{i})) dx \\
& = &
\int_\Omega \sum_{i = 1}^{N}k_{B}C_{i}\rho_{i} \frac{\partial \log  T}{\partial t}+\int_\Omega \sum_{i = 1}^{N}k_{B}C_{i}(\nabla(\rho_{i}u_{i})-\rho_{i}u_{i}\frac{\nabla T}{T}-u_{i}\nabla\rho_{i})+ \sum_{i = 1}^{N}k_{B} (1-C_{i}) (\frac{\rho_{i}}{\partial t}-\frac{\rho_{i}}{\partial t} \log  \rho_{i}) dx \\
& = & \int_\Omega \sum_{i = 1}^{N}k_{B}C_{i}\rho_{i} \frac{\partial \log  T}{\partial t}+\int_\Omega \sum_{i = 1}^{N}k_{B}C_{i}(-\frac{\rho_{i}}{\partial t}+\frac{\rho_{i}}{\partial t}\log  T--\frac{\rho_{i}}{\partial t} \log  \rho_{i})+ \sum_{i = 1}^{N}k_{B} (1-C_{i}) (\frac{\rho_{i}}{\partial t}-\frac{\rho_{i}}{\partial t} \log  \rho_{i}) dx \\
& = & \frac{d}{dt}\sum_{i = 1}^{N}\int_\Omega k_{B} \rho_i(\log  \rho_i -C_i \log T-C_i) dx. 
\end{eqnarray*}
Next, for the right hand side we have:
\begin{eqnarray*}
\lefteqn{\int_\Omega \frac{\nabla\cdot k\nabla T}{T}+\sum_{i = 1}^{N} \frac{\nu_{i}\rho_i |u_i|^2}{T}+\frac{q}{T} dx} \\
&= & -\sum_{i = 1}^{N} \int_{\Omega} \frac{\nu_{i}\rho_i |u_i|^2}{T}+\frac{1}{k}|\frac{k\nabla T}{T}|^2  - \nabla\cdot(\frac{k\nabla T}{T})
\end{eqnarray*}
Notice that in the derivation we used the continuity equations.
Also, since this energy law is for closed systems, zero flux boundary condition is used when integrating by parts and finally we used that the heat source vanishes, namely,  we have $q=0$. 

Clearly, the energy law \eqref{energy_law_T} makes sense
when  ion concentrations and temperature are positive
and the solution satisfies the following regularity assumptions:
\begin{eqnarray*}
  &&\phi \in \mathbb{H}^{1}_{\Gamma_{D}} \equiv \{ \nu\in \mathbb{H}^{1}(\Omega) |  \left.\nu\right\vert_{\Gamma_{D}} = \delta V   \}\\
  && \rho_i \in {W} \equiv \mathbb{H}^{1} \cap \mathbb{L}^{\infty}(\Omega)
  \end{eqnarray*}

\subsection{Log-density formulation and its energy}
We introduce a change of variables (also known as a $\log$  transformation for the
ion concentrations) which is as follows

\begin{align*}
\eta_{i}(x,t) &= \log  \rho_{i}(x,t)  \text {,             } i = 1,...,N\\
\xi(x,t) & = \log  T(x,t) 
\end{align*}
We note that such change of variables requires that the concentrations are positive and we have 
\[\eta_i = \log \rho_i \in W \equiv \mathbb{H}^{1} \cap \mathbb{L}^{\infty}(\Omega) \hookrightarrow \mathbb{H}^{2}
\]
The embedding above holds for space dimensions $d \le 3$.  The
Heat-PNP system then is written in a variational (weak) form: Find
$\eta_{i}(t) \in W$ , $ \xi(t) \in W$ with and $\phi \in
\mathbb{H}^{1}_{\Gamma_{D}} $ such that:
\begin{equation}\label{variational_form}
\begin{cases}
\displaystyle  (\epsilon\nabla\phi,\nabla v) + \langle \kappa \phi ,v \rangle_{\Gamma_{R}}- \sum_{i = 1}^{N} q_{i} (e^{\eta_{i}},v) = \langle C,v \rangle_{\Gamma_{R}}+ \langle S,v\rangle_{\Gamma_{N}}\\
\displaystyle (\frac{\partial}{\partial t} e^{\eta_{i}},w) + (\frac{\nabla(e^{\eta_{i}}e^{\xi})+e^{\eta_{i}}z_{i}e\nabla\phi}{\nu_{i}},\nabla w) = 0\\
\displaystyle (\sum_{i = 1}^{N} k_BC_i e^{\eta_{i}} \frac{\partial}{\partial t}e^{\xi},w) - (\sum_{i = 1}^{N} k_BC_i e^{\xi} u_i,\nabla( e^{\eta_{i}}w)) + (k\nabla e^{\xi},\nabla w) = (\sum_{i=1}^N \nu_i e^{\eta_{i}}|u_i|^2,w)
\end{cases}
\end{equation}
Here we denote by $\vec{u}_i$
the macroscopic velocity of ith species defined as
$\vec{u}_i = \frac{\nabla(e^{\eta_{i}}e^{\xi})+e^{\eta_{i}}z_{i}e\nabla\phi}{\nu_{i}e^{\eta_{i}}}$, and for the energy law~\eqref{energy_law_T} we have
\begin{equation}\frac{d}{dt}\sum_{i = 1}^{N}\int_\Omega k_{B} e^\eta_i(\eta_i -C_i \xi-C_i) dx = -\sum_{i = 1}^{N} \int_{\Omega} \frac{\nu_{i}e^{\eta_{i}} |u_i|^2}{T}+\frac{1}{k}|\frac{k\nabla T}{T}|^2  - \nabla\cdot(\frac{k\nabla T}{T})
\end{equation}
Furthermore, if we denote $\nu_{i}=1$, $C_{i}=1$, $z_{i}=1$ and for
simplicity we ignore the constant $k_{B}$, and under the appropriate
zero flux boundary condition we arrive at the following simplified
form of~\eqref{energy_law_T}:
\begin{equation}\label{energy_law_T_weak}  
  \frac{d}{dt}\sum_{i = 1}^{N}\int_\Omega e^\eta_i(\eta_i -\xi -1) dx = -\sum_{i = 1}^{N} \int_{\Omega} \frac{e^{\eta_{i}} |u_i|^2}{e^{\xi} }+k|\nabla \xi^2|.
\end{equation}

The derivation of the the energy law~\eqref{energy_law_T_weak} from
the variational form~\eqref{variational_form} is as follows:
We choose the test function to be $e^{-\xi}$ in the temperature
equatio and we obtain
\begin{align*}
(\sum_{i = 1}^{N} e^{\eta_{i}} \frac{\partial}{\partial t}e^{\xi},e^{-\xi}) - (\sum_{i = 1}^{N}e^{\xi} u_i,\nabla( e^{\eta_{i}}e^{-\xi})) + (k\nabla e^{\xi},\nabla e^{-\xi}) &= (\sum_{i=1}^N  e^{\eta_{i}}|u_i|^2,e^{-\xi})\\
\int_{\Omega} \sum_{i = 1}^{N} e^{\eta_{i}} \frac{\partial \xi }{\partial t} - (e^{\xi}(\nabla\eta_{i}+\nabla\xi)+z_{i}e\phi,e^{-\xi}e^\eta_{i}(\nabla\eta_{i}-\nabla\xi))-k|\nabla \xi^2|  & = \sum_{i = 1}^{N} \int_{\Omega} \frac{e^{\eta_{i}} |u_i|^2}{e^{\xi} } \\
 \int_{\Omega} \sum_{i = 1}^{N} e^{\eta_{i}} \frac{\partial \xi }{\partial t} - (e^{\xi}e^\eta_{i}(\nabla\eta_{i}+\nabla\xi)+ez_{i}e^\eta_{i}\phi,(\nabla\eta_{i}-\nabla\xi))-k|\nabla \xi^2|  & = \sum_{i = 1}^{N} \int_{\Omega} \frac{e^{\eta_{i}} |u_i|^2}{e^{\xi} } \\
\end{align*}
Next, we test the continuity equation with $w = \eta_{i}-\xi$:  
\begin{align*}
(\frac{\partial}{\partial t} e^{\eta_{i}}, \eta_{i}-\xi) + ({\nabla(e^{\eta_{i}}e^{\xi})+e^{\eta_{i}}z_{i}e\phi},\nabla ( \eta_{i}-\xi )) &= 0\\
(e^{\eta_{i}}\frac{\partial}{\partial t}\eta_{i} , \eta_{i}-\xi) &= -(e^{\eta_{i}}e^{\xi}{\nabla({\eta_{i}}+{\xi})+e^{\eta_{i}}z_{i}e\phi},\nabla ( \eta_{i}-\xi )) 
\end{align*}
Combining the temperature and the continuity equations together shows that:
\begin{align*}
\int_{\Omega} \sum_{i = 1}^{N} e^{\eta_{i}} \frac{\partial \xi }{\partial t}+(e^{\eta_{i}}\frac{\partial}{\partial t}\eta_{i} , \eta_{i}-\xi)& = - \sum_{i = 1}^{N} \int_{\Omega} \frac{e^{\eta_{i}} |u_i|^2}{e^{\xi} } +k|\nabla \xi^2|  \\
\int_{\Omega} \sum_{i = 1}^{N} \frac{d}{dt} e^{\eta_{i}}(\eta_{i}-\xi-1)& = - \sum_{i = 1}^{N} \int_{\Omega} \frac{e^{\eta_{i}} |u_i|^2}{e^{\xi} } +k|\nabla \xi^2|  \\
\end{align*}

\subsection{The discrete formulation}
We consider a mesh $\mathbb{T}_{h}$ of simplices covering our computational domain (triangles in2D or tetrahedra in 3D).
As approximating space we
take the  space for piecewise linear, with respect to $\mathbb{T}_h$, continuous polynomials \cite{Brenner:2008},
\[
W_{h} \equiv \{w_{h} \in \mathbb{H}^{1} | w_{h}|_{\tau} \in \mathbb{P}^{1} \text{ for all  }\tau \text{ in } \mathbb{T}_{h}  \} \subset \mathbb{H}^{1} $$
$$ V_{h,\Gamma_{D}} \equiv \{v_{h} \in W_{h}| v|_{\Gamma_{D}}=h|_{\Gamma_{D}}  \}
\]
The finite element solution to the Heat-PNP system then
is defined using these finite element spaces. As a time marching scheme we choose the
backward Euler scheme which is implicit and therefore stable. 
The discrete variational form then is: Find $ e^{\eta^{j}_{i,h}} \in W_h , \xi^{j}_{h} \in W_h$ and
$\phi_{h}^{j} \in V_{h,\Gamma_{D}}$
\begin{equation}\label{discrete}
\begin{cases}
\displaystyle  (\epsilon\nabla\phi_{h}^{j},\nabla v_{h}) + \langle \kappa \phi_{h}^{j} ,v_h \rangle_{\Gamma_{R}}- \sum_{i = 1}^{N} q_{i} (e^{\eta^{j}_{i,h}},v_h) = \langle C,v_h \rangle_{\Gamma_{R}}+ \langle S,v_h\rangle_{\Gamma_{N}} \\
\displaystyle \frac{1}{\Delta t_{j}}(e^{\eta^{j}_{i,h}},w_h) + (\frac{\nabla(e^{\eta^{j}_{i,h}}e^{\xi^{j}_{h}}+q_{i}\phi_{h}^{j})}{\nu_{i}},\nabla w_h) = \frac{1}{\Delta t_{j}}(e^{\eta^{j-1}_{i,h}},w_h)  \\
\displaystyle \frac{1}{\Delta t_{j}} (\sum_{i = 1}^{N} k_BC_i e^{\eta^{j}_{i,h}} e^{\xi^{j}_{h}}\xi^{j}_{h},w_h) - (\sum_{i = 1}^{N} k_BC_i e^{\xi^{j}_{h}} u_{i,h}^{j},\nabla(\eta^{j}_{i,h},w_h)) + (k\nabla e^{\xi^{j}_{h}},\nabla w_h) = (\sum_{i=1}^N \nu_i e^{\eta^{j}_{i,h}} |u_{i,h}^{j}|^2,w_h) +  \\
 \frac{1}{\Delta t_{j}} (\sum_{i = 1}^{N} k_BC_i e^{\eta^{j-1}_{i,h}} e^{\xi^{j}_{h}}\xi^{j-1}_{h},w_h)
  \end{cases}
 \end{equation}
The approximation to the initial conditions uses the standard interpolation
operator $I_{h}$ and is as follows:
\begin{eqnarray*}
  &&\eta_i(x,0) = I_{h}(\log (\rho_{i,0}(x))) \text{ for x }\in  \Omega,\text{i = 1,...,N}\\
  && \xi(x,0) = I_{h}(\log(T_{0}(x)))  \text{ for x }\in  \Omega,
\end{eqnarray*}

\subsection{A discrete energy estimate}
Next result shows a discrete energy estimate which holds in case  when the mesh is quasi-uniform mesh and the stiffness matrix on each
Picard iteration is an $M$-matrix. In this case,
we can show that the corresponding discrete solution $ T_{h}^{j}$ satisfies maximum principle and with this condition satisfied, we have discrete energy estimate and mass conservation. 
\begin{thm}\label{t:energy_estimate}
  Suppose $\eta_{i,h}^{j} \in W_{h}$ and
  $\phi_{h}^{j} \in V_{h,\Gamma_{D}}$ satisfy equations \eqref{discrete} for $i =1 ,..,N$
and the assumptions for the discrete maximum principle of temperature equation
are satisfied. Then the mass is conserved for each ion species
\begin{equation}
\int_{\Omega} e^{\eta^{j}_{i,h}(x,t)}dx = \int_{\Omega} e^{\eta^{0}_{i,h}(x,t)}dx,\text{ for } i = 1,..,N , j = 1,...m.
\label{discrete_mass_conservation}
\end{equation}
Moreover, the discrete analogue of the energy estimate
(second law of thermodynamics) holds
\begin{eqnarray*}
  \lefteqn{\sum_{i = 1}^{N}\int_\Omega e^{\eta^{j}_{i,h}}( \eta^{j}_{i,h} -\xi^{j}_{h}-1) dx}\\
  && + 
\sum_{j = 1}^{m}\Delta t_{j}\sum_{i = 1}^{N} \int_{\Omega} -\frac{e^{\eta^{j}_{i,h}} |u_i|^2}{e^{\xi^{j}_{h}}}+|k\nabla e^{\xi^{j}_{h}}|^2 \cdot(1+Mh^{\frac{1}{2}}) \\
&\leq &
\sum_{i = 1}^{N}\int_\Omega  e^{\eta^{0}_{i,h}}( \eta^{0}_{i,h} -\xi^{0}_{h}-1) dx 
\end{eqnarray*}
where $\|e^{-\xi^{j}_{h}}\|_{W^{2,\infty}} \leq M$ and h is the mesh size.
\end{thm}
\begin{proof}
For simplicity, let $\nu_{i}=1$, $C_{i}=1$, $z_{i}=1$, $q=1$, and
$k_{B}=1$, and the corresponding discrete energy estimate is done as
follows. First, we choose $w_{h} \equiv 1 \in W_{h}$ in the first
equation of (\ref{discrete}) and this gives us:
\[ \frac{1}{\Delta t_{j}}\int_{\Omega} e^{\eta^{j}_{i,h}} -e^{\eta^{j-1}_{i,h}} = 0
\]
For the energy estimate, we test the last equation in (\ref{discrete}) with $w_{h} = I_h(e^{-\xi^{j}_{i,h}}) \in W_{h} $. The latter is
a valid test function because $-\xi^{j}_{i,h} \in W_{h}$ and $I_{h}$ is the interpolation operator mapping to
to the piecewise linear, continuous $W_h$.This gives us:
\begin{eqnarray*}
   &&  \underbrace{\frac{1}{\Delta t_{j}}
    \left(\sum_{i = 1}^{N} e^{\eta^{j+1}_{i,h}} e^{\xi^{j}_{h}}\xi^{j+1}_{h},I_{h}(e^{-\xi^{j}_{h}})\right)
    -\frac{1}{\Delta t_{j}} \left(\sum_{i = 1}^{N} e^{\eta^{j}_{i,h}} e^{\xi^{j}_{h}}\xi^{j}_{h},I_{h}(e^{-\xi^{j}_{h}})\right)}_{\text{I}} \\
  &&- \underbrace{(\sum_{i = 1}^{N} e^{\xi^{j}_{h}} u_{i,h}^{j},\nabla(e^{\eta^{j}_{i,h}}I_{h}(e^{-\xi^{j}_{h}})))}_{\text{II}}\\
  && + \underbrace{(k\nabla e^{\xi^{j}_{h}},\nabla I_{h}(e^{-\xi^{j}_{h}}))}_{\text{III}} \\
  &=&\underbrace{(\sum_{i=1}^N \nu_i e^{\eta^{j}_{i,h}} |u_{i,h}^{j}|^2,\frac{1}{e^{\xi^{j}_{h}}})}_{\text{IV}}
\end{eqnarray*}

We analyze each term in the above equation sum and we begin by rewriting 
$\text{I}$  on the left hand side. We have
\begin{eqnarray*}
\text{I}& = &\frac{1}{\Delta t_{j}} (\sum_{i = 1}^{N} e^{\eta^{j+1}_{i,h}} e^{\xi^{j}_{h}}\xi^{j+1}_{h},I_{h}(e^{-\xi^{j}_{h}}))-\frac{1}{\Delta t_{j}} (\sum_{i = 1}^{N} e^{\eta^{j}_{i,h}} e^{\xi^{j}_{h}}\xi^{j}_{h},I_{h}(e^{-\xi^{j}_{h}}))\\
& = & \frac{1}{\Delta t_{j}}((\sum_{i = 1}^{N} e^{\eta^{j+1}_{i,h}},\xi^{j+1}_{h})-(\sum_{i = 1}^{N} e^{\eta^{j}_{i,h}},\xi^{j}_{h}))(I_{h}(e^{-\xi^{j}_{h}})e^{\xi^{j}_{h}})
\end{eqnarray*}
Next, we rewrite also $\text{II}$ using integration by parts and
the zero flux boundary conditions:
\begin{eqnarray*}
  \text{II}&=&-(\sum_{i = 1}^{N} e^{\xi^{j}_{h}} u_{i,h}^{j},\nabla(e^{\eta^{j}_{i,h}}I_{h}(e^{-\xi^{j}_{h}})))\\
  &=& (\sum_{i = 1}^{N} \nabla (e^{\xi^{j}_{h}} u_{i,h}^{j})e^{\eta^{j}_{i,h}}I_{h}(e^{-\xi^{j}_{h}}))\\
& =&  (I_{h}(e^{-\xi^{j}_{h}})e^{\xi^{j}_{h}})\cdot((\sum_{i = 1}^{N} u_{i,h}^{j} e^{\eta^{j}_{i,h}},\nabla \xi^{j}_{h})+(\sum_{i = 1}^{N} \nabla u_{i,h}^{j} e^{\eta^{j}_{i,h}},1) )\\
& =&(I_{h}(e^{-\xi^{j}_{h}})e^{\xi^{j}_{h}})\cdot( \frac{1}{\Delta t_{j}} (e^{\eta^{j+1}_{i,h}},\eta^{j+1}_{i,h}-\xi^{j}_{h}-1)- \frac{1}{\Delta t_{j}} (e^{\eta^{j}_{i,h}}, \eta^{j}_{i,h}-\xi^{j}_{h}-1))
\end{eqnarray*}
In the last step we used the discretized continuity equation to estimate
term $\text{III}$ on the left hand side as follows:
\begin{eqnarray*}
  \text{III}&=&(k\nabla e^{\xi^{j}_{h}},\nabla I_{h}(e^{-\xi^{j}_{h}}))\\
  &=& k|\nabla e^{\xi^{j}_{h}}|^2-(k\nabla e^{\xi^{j}_{h}},\nabla I_{h}(e^{-\xi^{j}_{h}})-\nabla e^{-\xi^{j}_{h}})\\
&\leq & k|\nabla e^{\xi^{j}_{h}}|^2+k\|\nabla e^{\xi^{j}_{h}}\|_2\cdot \| \nabla I_{h}(e^{-\xi^{j}_{h}})-\nabla e^{-\xi^{j}_{h}} \|_{2}
\end{eqnarray*}
Here we have used the
assumption that $e^{-\xi^{j}} \in W^{2,\infty}$, $e^{-\xi^{j}_{h}} \in W^{2,\infty}$ for any $j$ and that these constants
are uniformly bounded by the constant $M$.
As shown in \cite[Theorem 3.1.6]{Ciarlet:2002} we have the following interpolation estimate
\begin{align*}
\|\nabla I_{h}(e^{-\xi^{j}_{h}})-\nabla e^{-\xi^{j}_{h}}\|_{L^{2}} & \leq Ch | e^{-\xi^{j}_{h}} | _{2,\Omega} = C(\Omega)hM.
\end{align*}
term4 on the right hand side:
\begin{align*}
\left(\sum_{i=1}^N e^{\eta^{j}_{i,h}} |u_{i,h}^{j}|^2,\frac{1}{e^{\xi^{j}_{h}}}\right) = \left(\frac{\sum_{i=1}^N e^{\eta^{j}_{i,h}} |u_{i,h}^{j}|^2}{e^{\xi^{j}_{h}}}\right)(I_{h}(e^{-\xi^{j}_{h}})e^{\xi^{j}_{h}})\end{align*}
Combining all the terms above and by a discrete Gr$\ddot{\o}$nwall argument of telescopic sum from time step t=0 to m and, as a result, we obtain the following energy estimate:
\begin{eqnarray*}
\sum_{j=1}^m\Delta t_j
  \left(\text{I}+\text{II}\right)&\ge & \sum_{j=1}^m\Delta t_j \left(\text{IV}-(Ch| I_{h}(e^{-\xi^{j}_{h}})- e^{-\xi^{j}_{h}} | _{2,\Omega}\cdot\|\nabla e^{\xi^{j}_{h}} \|_{2})\right)
\end{eqnarray*}
  only keep the term $\sum_{i = 1}^{N}\int_\Omega  e^{\eta^{0}_{i,h}}( \eta^{0}_{i,h} -\xi^{0}_{h}-1) dx $ in \text{I} to the right side, we have:
\begin{eqnarray*}
  \lefteqn{\sum_{i = 1}^{N}\int_\Omega e^{\eta^{j}_{i,h}}( \eta^{j}_{i,h} -\xi^{j}_{h}-1) dx }\\
  & + &
  \sum_{j = 1}^{m}\Delta t_{j}\sum_{i = 1}^{N} \int_{\Omega} -\frac{e^{\eta^{j}_{i,h}} |u_i|^2}{e^{\xi^{j}_{h}}}+|k\nabla e^{\xi^{j}_{h}}|^2 \cdot(\frac{1}{(I_{h}(e^{-\xi^{j}_{h}})e^{\xi^{j}_{h}})}) \\
  &+ &\frac{1}{(I_{h}(e^{-\xi^{j}_{h}})e^{\xi^{j}_{h}})}\cdot(Ch| I_{h}(e^{-\xi^{j}_{h}})- e^{-\xi^{j}_{h}} | _{2,\Omega}\cdot\|\nabla e^{\xi^{j}_{h}} \|_{2}) dx \\
&\le & \sum_{i = 1}^{N}\int_\Omega  e^{\eta^{0}_{i,h}}( \eta^{0}_{i,h} -\xi^{0}_{h}-1) dx 
\end{eqnarray*}

divide the equation by $(I_{h}(e^{-\xi^{j}_{h}})e^{\xi^{j}_{h}})$ on both side, use  $(I_{h}(e^{-\xi^{j}_{h}})e^{\xi^{j}_{h}}\ge 1$
and notice $I_{h}$ is a piecewise linear interpolant of a convex
(the exponential) function shows that we can drop the term
$(I_{h}(e^{-\xi^{j}_{h}})e^{\xi^{j}_{h}}$ from the equation above and obtain the final estimate:

\begin{eqnarray*}
  \lefteqn{\sum_{i = 1}^{N}\int_\Omega e^{\eta^{j}_{i,h}}( \eta^{j}_{i,h} -\xi^{j}_{h}-1) dx}\\
  &&+ \sum_{j = 1}^{m}\Delta t_{j}\sum_{i = 1}^{N} \int_{\Omega} -\frac{e^{\eta^{j}_{i,h}} |u_i|^2}{e^{\xi^{j}_{h}}}
  +|k\nabla e^{\xi^{j}_{h}}|^2 \cdot(\frac{1+ChM }{(I_{h}(e^{-\xi^{j}_{h}})e^{\xi^{j}_{h}})}) dx \\
  & \leq &
\sum_{i = 1}^{N}\int_\Omega  e^{\eta^{0}_{i,h}}( \eta^{0}_{i,h} -\xi^{0}_{h}-1) dx.
\end{eqnarray*}
\end{proof}

In order to show that our numerical model meets the criteria of
Theorem~\ref{t:energy_estimate}, we need to show thatthe discrete
solution for the nonlinear temperature equation 
has an $L^2$ norm which is bounded below,
uniformly with respect to time. To see this, observe that 
on each time step, we are solving a nonlinear drift diffusion
equation for the temperature. We linearize these equations
using a Picard iteration. The solutions to the Picard iteration
satisfy the discrete maximum principle because we used EAFE discretization
for the linearized equations. 

\subsection{Fixed-Point Iteration}
We next write out the Picard  iteration algorithm for the temperature equation
which we solve each time step. 
\begin{algorithm}
\caption{Fixed Point Iteration for temperature equation}
\begin{algorithmic}[1]
\For{j = 1,2,3,...T}{\Comment time iteration}
    \State Get the solution from last time step: $\xi_{h}^{j-1},\eta^{j-1}_{i,h},\phi^{j-1}_{h}$;
    \State Choose a small number $\epsilon$ and set $\xi^{0} = \xi_{h}^{i-1}$\;
    \For {k = 1,2,3,...maxiter}{ \Comment{Non-linear picard iteration:}
    	\State Solve the linearized concentration equation and the Poisson equation using the temperature term from last nonlinear iteration step $\xi^{k-1}$;
     	\State Solve the linearized equation for temperature with the given dirichlet boundary condition\;
     
     	\If{ $\|\xi^{k}-\xi^{k-1} \|_{2} \leq \epsilon$ \textbf{and} $\|\eta_{i}^{k}-\eta_{i}^{k-1}\|_{2} \leq \epsilon $ for all i   } \Comment{Convergence Criteria}
     		\State {Stop}
    	\Else
    		\State {k=k+1}
    	\EndIf
	\EndFor
\EndFor
}
\end{algorithmic}
\end{algorithm}
Here, the linearized concentration equations are:
\begin{align*}
\displaystyle \frac{1}{\Delta t_{j}}(e^{\eta^{j,k}_{i,h}},w_h) + (\frac{\nabla(e^{\eta^{j,k}_{i,h}}e^{\xi^{j,k-1}_{h}}+q_{i}\phi_{h}^{j,k-1})}{\nu_{i}},\nabla w_h) = \frac{1}{\Delta t_{j}}(e^{\eta^{j-1}_{i,h}},w_h) \text{  for i  = 1,2,...N } \\
\end{align*}

The linearized temperature equation is :
\begin{eqnarray*}
\lefteqn{\frac{1}{\Delta t_{j}} (\sum_{i = 1}^{N} k_BC_i e^{\eta^{j}_{i,h}} e^{\xi^{j,k}_{h}},w_h) - (\sum_{i = 1}^{N} k_BC_i e^{\xi^{j}_{h}} u_{i,h,k-1}^{j},\nabla(\eta^{j}_{i,h},w_h)) + (k\nabla e^{\xi^{j}_{h}},\nabla w_h)}\\
  &=& (\sum_{i=1}^N \nu_i e^{\eta^{j}_{i,h}} |u_{i,h,k-1}^{j}|^2,w_h)
+ \frac{1}{\Delta t_{j}} (\sum_{i = 1}^{N} k_BC_i e^{\eta^{j-1}_{i,h}} e^{\xi^{j-1}_{h}},w_h)
\end{eqnarray*}
here, $ u_{i,h,k-1}^{j} $ has the meaning of we use the temperature from $(k-1)$-st fixed point iteration step to compute the discretized velocity term. The superscript $k$ stands for the nonlinear iteration step. 

\subsection{Discrete Maximum Principle with EAFE stabilization}

For the sake of clarity, we change the variables back and use the
original density and temperature variable. It is easy to see that the
corresponding continuous temperature equation in each Picard iteration
has this form:
\begin{eqnarray*}
  \lefteqn{\frac{1}{\Delta t}\left(\sum_{i=1}^N \rho_{i,h}^{j,k}\right) T_{h}^{j,k} + \left(\sum_{i=1}^N  \rho_{i,h}^{j.k} \vec{u_{i,h}^{j,k-1}} \right) \cdot \nabla T_{h}^{j,k} + \left(\sum_{i=1}^N \rho_{i,h}^{j,k}  \nabla \cdot \vec{u_{i,h}^{j,k-1}}\right) T_{h}^{j,k}}\\
    & = & \nabla \cdot k \nabla T_{h}^{j,k} +  \sum_{i=1}^N \nu_i\rho_{i,h}^{j,k} |\vec{u_{i,h}^{j,k}}|^2+ \frac{1}{\Delta t}\left(\sum_{i=1}^N \rho_{i,h}^{j,k}\right) T_{h}^{j,k-1} 
\end{eqnarray*}
This can be written in a  more concise form:
\[
L T_{h}^{j,k} =  \sum_{i=1}^N \nu_i\rho_{i,h}^{j,k} |\vec{u_{i,h}^{j,k}}|^2,
\]
where the differential operator for the temperature equation:
\[
Lu=-k\sum_{i,j=1}^{n}u_{x_{i}x_{j}}+\sum_{i=1}^{n}b^{i}u_{x_{i}}+cu
\]
and $c=\frac{1}{\Delta t}-\left(\sum_{i=1}^N \rho_{i,h}^{j,k} \nabla
\cdot \vec{u_{i,h}^{j,k-1}}\right)$, $ b^{i} = \left(\sum_{i=1}^N
\rho_{i,h}^{j,k} \vec{u_{i,h}^{j,k-1}} \right)^{i} $.  It is crucial
that the time step is chosen so that the parameter $c$ in front of the
temperature term is always positive because this implies that the
linearized differential operator is uniformly elliptic. Also the right
hand side is always positive, and hence, 
\[
L T_{h}^{j,k} \geq 0  \text{      in    } \Omega,
\]
If we assume that the linearized
temperature equation
has a classical solution in $C^{2}(\Omega) \cap C(\bar{\Omega})$, according to~\cite{Evans:2010} our solution  satisfies weak maximum principle\cite{Evans:2010}
 and the minimum of the temperature
 must be attained at the boundary.
 In applications the this Dirichlet boundary condition is
 always positive which implies that
 we have a uniform upper bound on $\|T_{h}^{j}\|_{2}$ for all time steps.

 In order to meet the criteria for the discrete energy estimate in
 Theorem~\ref{t:energy_estimate}, the inverse of discretized
 temperature solution from the nonlinear solver at each time step need
 to satisfy a global upper bound in the $L^2$ sense.  This requires a
 special numerical treatment of the temperature equation. Inspired by
 techniques introduced in a recent work~\cite{Liu:2015}, we maintain a
 discrete maximum principle for the advection diffusion equation with
 source term in the following way: To do this we rewrite the
 linearized equations for temperature in divergence form and then use
 the EAFE scheme to guarantee that the discrete maximum principle is
 satisfied and the maximum of the temperature is on the boundary.
\begin{equation}
\nabla(-k\nabla T_{h}^{j,k}+ \sum_{i=1}^N \rho_{i,h}^{j,k}  \vec{u_{i,h}^{j,k-1}}T_{h}^{j,k})+c T_{h}^{j,k} = G(\vec{u_{i,h}^{j,k-1}},\rho_{i,h}^{j,k},T_{h}^{j,k-1}).
\end{equation}
here, $c=c(\Delta t, \rho_{i,h}^{j,k},\vec{u_{i,h}^{j,k-1}} ) = \frac{1}{\Delta t}-\left(\sum_{i=1}^N \nabla \rho_{i,h}^{j,k}   \vec{u_{i,h}^{j,k-1}}\right)$,  and $G$ is always positive. The corresponding weak form for the linearized equation is 
\begin{equation}
(k\nabla T_{h}^{j,k}+ \sum_{i=1}^N \rho_{i,h}^{j,k}  \vec{u_{i,h}^{j,k-1}}T_{h}^{j,k},\nabla w)+ (cT_{h}^{j,k},w) = (G(\vec{u_{i,h}^{j,k-1}},\rho_{i,h}^{j,k},T_{h}^{j,k-1}),w)
\end{equation}
We apply mass lumping to the term $(cT_{h}^{j,k},w)$ , using nodal interpolant $I_{h}: W \rightarrow W_{h}$,
$$(cT_{h}^{j,k},w)_{h} = \int_{\Omega} I_{h}((\frac{1}{\Delta t}-\left(\sum_{i=1}^N \nabla \rho_{i,h}^{j,k}   \vec{u_{i,h}^{j,k-1}}\right) )T_{h}^{j,k} )I_{h}(w)dx $$
and EAFE approximation to the flux term:
\begin{align*}
a_{h}(k\nabla T_{h}^{j,k}+ \sum_{i=1}^N \rho_{i,h}^{j,k} \vec{u_{i,h}^{j,k-1}}T_{h}^{j,k},w ):=&\sum_{\tau \in \mathrm{T}_{h}}\sum_{E \in \tau} \omega^{\tau}_{E}\frac{|E|}{\int_{E}e^{c(\Delta t, \rho_{i,h}^{j,k},\vec{u_{i,h}^{j,k-1} )}ds}}\delta_{E}(e^{c(\Delta t, \rho_{i,h}^{j,k},\vec{u_{i,h}^{j,k-1} })}T_{h}^{j,k})\delta_{E}(w) \\
&\approx (k\nabla T_{h}^{j,k}+ \sum_{i=1}^N \rho_{i,h}^{j,k}  \vec{u_{i,h}^{j,k-1}}T_{h}^{j,k},\nabla w).
\end{align*}
The changes in notation used above are summarized as follows:
\begin{itemize}
\item $ \mathbb{T}_{h} $ is either the triangulation in 2d or tetrahedron in 3d.
\item $ \tau$ stands for the simplex 
\item  E stands for the edges
\item $\omega_{E}$ is the entries in the stiffness matrix for the Laplace equation. 
\end{itemize}
The detailed derivation and proof of monotonicity as well as the restrictions on the mesh needed for such a proof are found in \cite{Xu:1999}. 

As a result we have the following theorem providing
discrete energy estimate. 
\begin{thm}
Suppose that the Dirichlet boundary condition for temperature are
strictly positive and no-flux boundary condition is imposed for each
ion density our nonlinear Picard iteration approach using EAFE scheme
with appropriate small time step satisfies the discrete energy
estimate:
\begin{eqnarray*}
  \lefteqn{\sum_{i = 1}^{N}\int_\Omega e^{\eta^{j}_{i,h}}( \eta^{j}_{i,h} -\xi^{j}_{h}-1) dx}\\
  && +
\sum_{j = 1}^{m}\Delta t_{j}\sum_{i = 1}^{N} \int_{\Omega} \frac{e^{\eta^{j}_{i,h}} |u_i|^2}{e^{\xi^{j}_{h}}}
+|k\nabla e^{\xi^{j}_{h}}|^2 \cdot(1+Mh^{\frac{1}{2}}) dx\\
& \leq & 
\sum_{i = 1}^{N}\int_\Omega  e^{\eta^{0}_{i,h}}( \eta^{0}_{i,h} -\xi^{0}_{h}-1) dx 
\end{eqnarray*}
where $\|e^{-\xi^{j}_{h}}\|_{W^{2,\infty}} \leq M$ and h is the mesh size.
and also the conservation of charge is satisfied:
\begin{equation}
\int_{\Omega} e^{\eta^{j}_{i,h}(x,t)}dx = \int_{\Omega} e^{\eta^{0}_{i,h}(x,t)}dx,\text{ for } i = 1,..,N , j = 1,...m
\end{equation}
\end{thm}


\section{Numerical experiments}
The first numerical example is modeling an electric device consists of two kinds of ionic solutions: Na and Cl, and compute the current , density and temperature profile under given fixed Voltage and outside temperature. With initial condition $\rho_\pm(x,0)=\rho_0=0.06$ and $T(x,0)=1$. The dimensionless parameters are, $C_0 \rho_0=302$, $C_\pm=3$, $1/\nu_+=1.334$, $1/\nu_-=2.032$ \cite{lide:2004}, $\epsilon=1$, $l_B=0.714$. In order to highlight the contribution from temperature, we choose a relatively small heat conductance $k=100$. 
The computational domain is $[0,10]\times[0,1]$. The boundary condition for ion density and temperature are Dirichlet, i.e. $\rho_i(0,t)=\rho_i(L,t)=\rho_0$, $T(0,t)=T(L,t)=1$.

\subsection{Demonstrating Discrete Energy Dissipation} 
Notice that we need zero boundary flux for each of the ion species in order to have charge (mass) conservation and discrete energy estimate. In the numerical experiment presented in this section we are imposing Dirichlet boundary conditions on ion densities, and our system is no longer a closed system. Thus, the energy estimate needs to take into consideration the boundary flux. In such setting, the exact conservation law of entropy is as follows: 
\begin{equation}
\frac{d }{dt}S(V,t) + \int_{\partial V} \frac{j}{T} \cdot d\mathbf{r} - \int_{V} \frac{q}{T}  d\mathbf{r} -J_S=\Delta(V,t)\geq 0,
\end{equation}
The numerical results suggest that after a short time period the system
tends to its equilibrium state and the bulk entropy is always
increasing and tends to a limit. The entropy production cancels with
the boundary flux of entropy when equilibrium is reached. This is
exactly what we expected (see Figure~\ref{f:energy-law}).
\begin{figure}[!h]
\begin{center}
\includegraphics[width=0.45\textwidth]{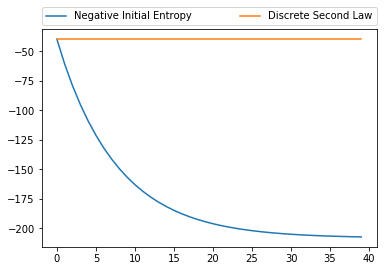}
\includegraphics[width=0.45\textwidth]{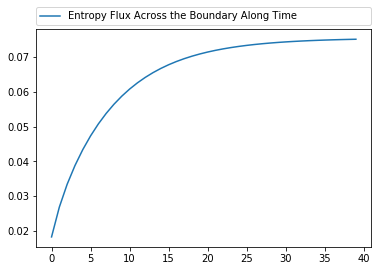}\\
\includegraphics[width=0.45\textwidth]{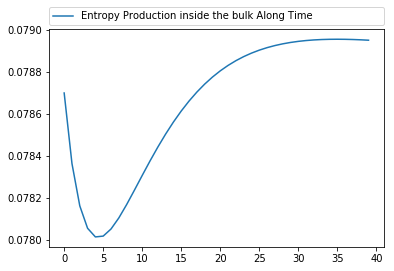}
\includegraphics[width=0.45\textwidth]{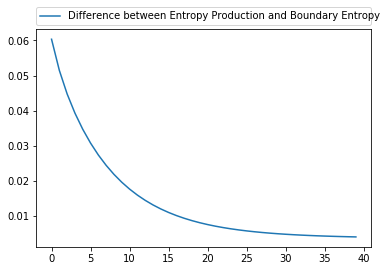}
\caption{When the system approaching to steady state. (a) Bulk Entropy is increasing and tends to a limit (b) Entropy Flux at the Boundary is decreasing to a limit (c) Dissipation is decreasing to a limit (d) The difference between Dissipation and Entropy Flux tends to zero\label{f:energy-law}}
\end{center}
\end{figure}

The steady state of PNP equation satisfies Poisson-Boltzmann (PB)
equation. But we can not directly compare our steady state solution to
PB solution since temperature effect is not considered there. To
validate the accuracy of our numerical method, we observe similar
steady-state solution with the PNP solution in previous
work\cite{Liu:2017}.

As is shown below, the temperature profile at steady-state is
concave. On the other hand the charged densities are convex. With
different increasing voltage by setting different electric potential
on each side of the tube, the current $I$  does not increase linearly as in
the non-isothermal case. This is more accurate model as  the increase
in voltage changes the temperature and the device we model becomes a
non-linear device (see Figure~\ref{f:nonlinear-device}.

\begin{figure}[!h]
\begin{center}
\includegraphics[width=0.45\textwidth]{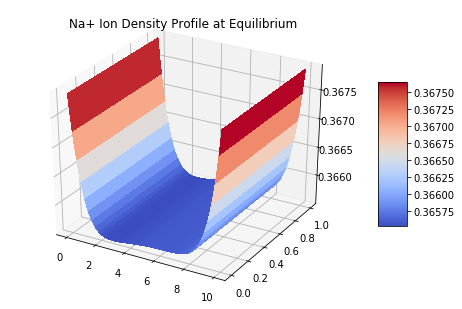}
\includegraphics[width=0.45\textwidth]{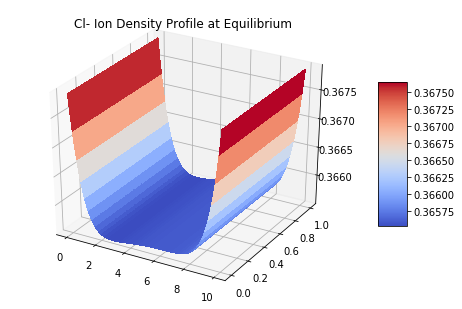}\\
\includegraphics[width=0.45\textwidth]{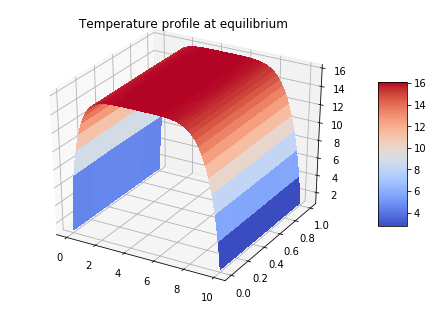}
\includegraphics[width=0.45\textwidth]{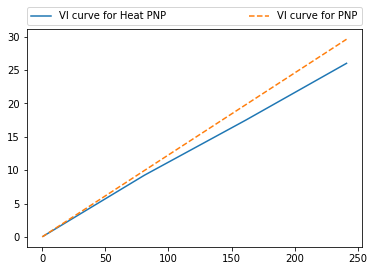}
\caption{When the system approaching to steady state. (a) $Na^+$ density distribution. (b)$Cl^-$ density distribution.(c) Temperature distribution. (d) Voltage-Current relation of the system.\label{f:nonlinear-device}}
\end{center}
\end{figure}

\subsection{Modeling electroosmosis flow}
We consider two ionic species with incompressible background fluid in an rectangular shape tube. 
Then the dynamic equations for the solute particles become,
\begin{equation}
\begin{cases}
\displaystyle \frac{\partial}{\partial t} \rho_i + \nabla \cdot (\rho_i u_i)=0 ,\\
\displaystyle m_i \rho_i (\frac{\partial u_i}{\partial t} + u_i \nabla  u_i)+\nabla P_i+\rho_i z_ie\nabla \phi=\nu_i \rho_i (u_0-u_i)+\nabla (\xi_i \nabla \cdot u_i) + \nabla \cdot \lambda_i \nabla u_i,\\
-\nabla \cdot \epsilon \nabla \phi =\sum_{m=1}^N \rho_m z_m e+\rho_f.
\end{cases}\hspace{-0.5cm}
\label{gPNP_solute}
\end{equation} 
Where $P_i$ is the thermodynamic pressure which we will be giving out later.
Here we use the fact that $-\nabla \cdot \epsilon \nabla v(\mathbf{r},\mathbf{r'}) = \delta(\mathbf{r}-\mathbf{r'})$, where $\epsilon$ is the dielectric constant. And $\rho_f = -\nabla \cdot \epsilon \nabla \psi$ describes the external field. 
For the incompressible solvent,
\begin{equation}
\begin{cases}
\displaystyle  m_0 \rho_0 \left(\frac{\partial}{\partial t} u_0 + u_0 \nabla  u_0\right)+\nabla P_0 +\rho_0 \nabla \phi_0=\sum_{i=1}^N\nu_i \rho_i (u_i-u_0) + \nabla \cdot \lambda_0 \nabla u_0,\\
  \nabla \cdot u_0 = 0.
\end{cases}
\label{gPNP_solvent}
\end{equation}

And the temperature equation,
\begin{align}
&\sum_{i=0}^N \left(- T \frac{\partial^2 \Psi_i}{\partial T^2}\right)\left( \frac{\partial T}{\partial t} +  u_i  \cdot \nabla T\right) + \left(\sum_{i=1}^N  \frac{\partial P_i}{\partial T}\nabla \cdot u_i \right)T  \nonumber\\=& \nabla \cdot k \nabla T +  \sum_{i=1}^N \nu_i \rho_i |u_i-u_0|^2 + \xi_i |\nabla \cdot u_i|^2 +\sum_{i=0}^N \lambda_i |\nabla u_i|^2 +q.
\label{gPNP_temp}
\end{align}

In~\eqref{gPNP_temp}, the term $ \sum_{i=0}^N \left(- T
\frac{\partial^2 \Psi_i}{\partial T^2}\right)$ can be viewed as the
weighted average heat capacitance of the system. The second term
represents the work of thermopressure transfer into heat. On the right
hand side, $ \nabla \cdot k \nabla T$ describes the heat diffusion. We
notice that the entropy production from mechanical viscosity appears as an
internal heat source.

Following the notation of \cite{Liu:2017}, the free energy of the whole system is:
\begin{eqnarray}
F(V,t)&=& \int_\Omega \Psi_0(\rho_0(\mathbf{r},t),T(\mathbf{r},t))+ \Psi_1(\rho_1(\mathbf{r},t),T(\mathbf{r},t))+ \Psi_2(\rho_2(\mathbf{r},t),T(\mathbf{r},t)) d\mathbf{r} \\
&&+\sum_{i,m=0}^N\frac{z_i z_m e^2}{2} \int_V \int_\Omega \rho_i(\mathbf{r},t) \rho_m(\mathbf{r'},t) v(\mathbf{r},\mathbf{r'}) d\mathbf{r}d \mathbf{r'}\nonumber\\
&&+ \sum_{i=0}^N z_i e \int_V \rho_i(\mathbf{r}) \psi(\mathbf{r},t) d \mathbf{r}.
\label{FreeEnergy_g}
\end{eqnarray}
Where: 
\begin{equation}
\Psi_0(\rho_0(\mathbf{r},t),T(\mathbf{r},t)) = T(\mathbf{r}) \log T(\mathbf{r})
\end{equation}
is the free energy term for the solvent.\\
$\Psi_i$,i = 1,2 which are the free energy terms for the two ionic species are given by:
\begin{equation}
\Psi_i(\rho_i(\mathbf{r},t),T(\mathbf{r},t)) = k_B T(\mathbf{r},t) \rho_i(\mathbf{r},t) \left[\log \rho_i(\mathbf{r},t) - C_i \log T(\mathbf{r},t)\right],
\label{ideal_free}
\end{equation}
The corresponding thermodynamic pressure will be:
\begin{equation}
\begin{cases}
\displaystyle P_{i} = \rho_{i} T,     i = 1,2\\
\displaystyle P_0 = -T \log T
\end{cases}
\end{equation}
and the whole system equations is as follows:
\begin{equation}
\begin{cases}
\displaystyle \frac{\partial}{\partial t} \rho_i + \nabla \cdot (\rho_i \vec{u_i})=0 ,      i = 0,1,2  \\
\displaystyle \nu_i \rho_i (\vec{u_i}-\vec{u_0}) = - k_B \nabla (\rho_i T) - z_i e \rho_i \nabla \phi,      i = 1,2  \\
\displaystyle m(\frac{\partial}{\partial t}\vec{u_{0}}+\vec{u_{0}}\nabla \vec{u_{0}})+ \nabla P_{0} + \sum_{i =1}^{2} \nu_i \rho_i (\vec{u_0}-\vec{u_i}) = \nabla\cdot\lambda_{0}\nabla \vec{u_0} \\
\displaystyle \nabla\cdot\vec{u_{0}} = 0\\
\displaystyle -\nabla \cdot \epsilon \nabla \phi =\sum_i \rho_i z_i e+\rho_f.\\
\displaystyle \left(\sum_{i=0}^2  k_B C_i \rho_i\right) \frac{\partial T}{\partial t} + \left(\sum_{i=0}^2  k_B C_i \rho_i \vec{u_i} \right) \cdot \nabla T + \left(\sum_{i=1}^2 k_B \rho_i  \nabla \cdot \vec{u_i}\right) T \nonumber =\\  \nabla \cdot k \nabla T +  \sum_{i=1}^N \nu_i\rho_i |\vec{u_i}-\vec{u_{0}}|^2 + \lambda_{0}|\vec{u_{0}}|^2+ q.
\end{cases}
\label{temp_Electroo}
\end{equation}\\~\\

We first apply asymmetrical boundary zeta potentials along the upper and lower wall of the channel and list the parameter values in Table~\ref{tab:zeta}.
\begin{table}[h]
  \centering
\begin{tabular}{ |p{4cm}||p{4cm}|p{4cm}|  }
 \hline
 \multicolumn{3}{|c|}{Parameter List} \\
 \hline
Parameter Name & Parameter Explanation &Value\\
 \hline
  $\phi\vert_{x=0}-\phi\vert_{x=10}$ & Voltage &100\\
  $\rho$ & background fluid density &1\\
 k & Heat conductance&  100\\
 Cv & Heat Capacitance & 300 \\
 z &  ASM &1\\
 e &  AND   &1\\
 $\mu$ & Mobility constant for background fluid& 1\\
 $\mu_1$ &Mobility constant for positive charge &1.334\\
 $\mu_2$ &Mobility constant for positive charge &2.032\\
 $\epsilon$ & Dielectric constant& 1\\
 $q_1$ & positive charge valence number &1 \\
 $q_2$ & negative charge valence number &-1\\
 \hline
\end{tabular}
\caption{\label{tab:zeta}Parameter values for the second numerical experiment.}
\end{table}

 The velocity plot of the positive charge species and the background fluid are shown in Figure~\ref{positive}, 
 \begin{figure}[H]
 \begin{center}
\includegraphics[width=0.9\textwidth]{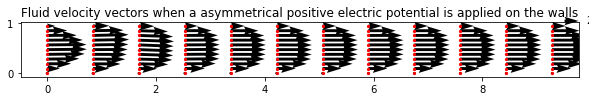}\\
\includegraphics[width=1.0\textwidth]{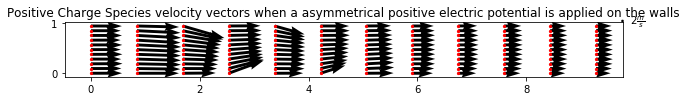}
\caption{Velocity plot of the positive charge species and the background fluid. (a) Background fluid velocity (b) Positive charge species velocity\label{positive}}
\end{center}
\end{figure}
Local Nusselt Number is used to describe the effect of zeta potentials on heat transfer which is defined by
\begin{equation}
Nu = \frac{-h\frac{\partial T}{\partial y}}{T_{w}-T_{m}}
\end{equation}
where h is the width of channel, $T_w$ is the wall temperature, $T_{m}$ is the bulk mean temperature, and y is the perpendicular distance from the channel walls. The Nusselt number along the upper and lower walls at stationary state is as shown in Figure~\ref{nusselt}.

 \begin{figure}[h]
 \begin{center}
\includegraphics[width=0.60\textwidth]{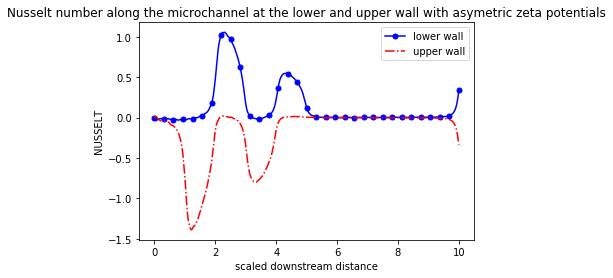}
\caption{Nusselt number along the upper and lower wall of channel.
  \label{nusselt}}

\end{center}
\end{figure}

We can see the zeta potential will cause perpendicular uneven distribution of the temperature, thus may have negative effect on the horizontal charge transport efficiency and also may have negative effect on the charge separation.\\~\\

We also computed the case where no zeta potential on the side walls are posed.In this case we examined the stationary temperature profile development and focusing on two aspects the radial and horizontal distribution.
\begin{figure}[H]
\begin{center}
\includegraphics[width=0.45\textwidth]{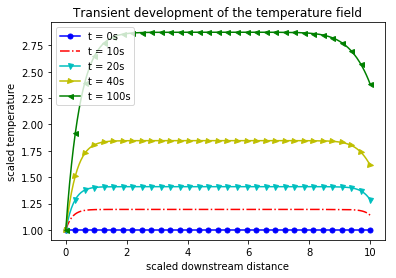}
\caption{The transient development of the temperature field}
\end{center}
\end{figure}

The above figure shows the transient development of the temperature field. The Joule heating effect is explicitly showed as time goes on, the temperature of the whole capillary tube has been elevated. The plot shows the temperature profile in horizontal direction of the center, and the temperature gradients are mainly in the inlet and outlet.

\begin{figure}[H]
\begin{center}
\includegraphics[width=0.45\textwidth]{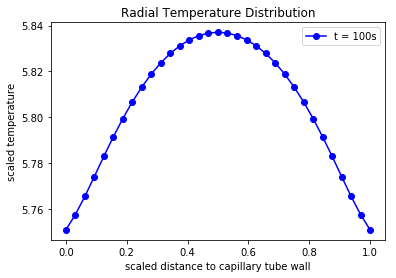}
\caption{radial temperature profile at equilibrium at the midpoint of channel.}

\end{center}
\end{figure}
Radial temperature profile has a parabola like shape, large radial temperature gradient may have negative effect on charge separation and electroosmosis process since the temperature gradient will induce charge motions orthogonal to axial direction.

In our model we captured the effect of temperature distribution
exhibiting a parabolic profile across the horizontal direction of the
tube. But the main influence on separation efficiency is via the
establishment of a radial temperature profile across the lumen of the
channel. An overall increase in temperature of the background fluid
has low influence on the overall quality of separation. It is known
that Joule heating parameter, auto thermal Joule heating parameter,
external cooling parameter, Peclet number are crucial in the process
of electroosmosis. Which can be translated into in our model. Our full
numerical experiment model will have the ability of using control
variable on various parameter and boundary conditions to examine the
effects from different factors on the quality of separation.

\section{Conclusions} 
We proposed a general framework for solving Non-Isothermal
electrokinetics equation based on a discretization using a logarithmic
transformation of the charge carrier densities and temperature
variable. We designed the numerical method for approximation of the
nonlinear fixed point iteration so that it meets the sufficient
conditions for strict discrete energy dissipation. The discrete energy
estimate, inherited from the continuous case, is satisfied and this
shows consistency of our numerical model with the thermodynamic laws.
Introducing more complex computational domains and coupling with
Navier Stokes equation allows for generalization of the numerical
models to electrochemistry and electrophysiology to study the heat
effect for battery, semiconductor and temperature gated ion
channels. Such topics are in the focus of our current and future
research.

\newpage
\bibliography{ref}
\bibliographystyle{ieeetr}

\end{document}